\documentclass[]{interact}
\usepackage{epstopdf}
\usepackage[caption=false]{subfig}
\usepackage{amsmath,amssymb}
\usepackage[utf8]{inputenc}
\usepackage[english]{babel}
\usepackage{natbib}
\bibpunct[, ]{[}{]}{,}{n}{,}{,}
\usepackage{hyperref}
\hypersetup{colorlinks=true, linkcolor=blue, citecolor=blue, urlcolor=blue}
\theoremstyle{plain}
\newtheorem{theorem}{Theorem}[section]
\newtheorem{lemma}{Lemma}[section]
\newtheorem{proposition}{Proposition}[section]
\newtheorem{corollary}{Corollary}[section]
\theoremstyle{definition}

\newtheorem{example}{Example}[section]
\theoremstyle{remark}
\newtheorem{remark}{Remark}[section]
\begin{document}
\articletype{ARTICLE}
\title{From Vieta to Chebyshev: A Unified Trigonometric Census of Polynomial Roots}
\author{
\name{Sawon Pratiher\textsuperscript{a}\thanks{\textsuperscript{a}Department of Electrical Engineering, Indian Institute of Technology Kharagpur, West Bengal 721302, India. (Email: \url{sawon1234@gmail.com}).}}
}
\maketitle
\begin{abstract}
This article gives a trigonometric method for counting how many roots of a polynomial are real and how many are complex, without solving the equation. For a depressed polynomial $P(t) = t^d + c_{d-2}t^{d-2} + \cdots + c_1 t + c_0$ of degree~$d$ with $c_{d-2} < 0$ (the case $c_{d-2}\geq 0$ is discussed briefly in Section~\ref{sec:cgeq0} but not resolved), the substitution $t = u\cos\theta$, $u = 2\sqrt{-c_{d-2}/d}$, matches $P$ against the Chebyshev identity $T_d(\cos\theta) = \cos d\theta$ and reduces the classification problem to a single trigonometric equation $f(\theta) = \sum_{k=1}^{d-3}\alpha_k\cos^k\!\theta + \cos d\theta + \delta = 0$ on $[0,\pi]$. A bijection between the zeros of~$f$ and the roots of~$P$ in $[-u,u]$ gives an exact interior root count. For roots outside $[-u,u]$, a universal lemma based on the positivity of~$P^{(d-2)}$ gives an exact parity and an upper bound of $d-2$ roots per half-line; combined with the interior count, this pins down the total number of real roots exactly whenever at most one root lies on each exterior half-line, a condition we prove for the cubic, verify computationally throughout our examples, but show by an explicit quartic counterexample to fail in general for $d\geq 4$ (Remark~\ref{rem:counterexample}). We also compare the method's cost to Sturm's theorem (Section~\ref{sec:complexity}): it is not asymptotically cheaper, but avoids the coefficient blow-up of the Euclidean algorithm and applies via one uniform template to every degree $d\geq 4$, with the classical Vieta cubic ($d=3$) as the degenerate base case. The derivation is conceptually simple, unifying and extending the quartic and quintic analyses developed individually in earlier work.
\end{abstract}
\begin{amscode}
12D10; 26C10; 33C45; 97H30
\end{amscode}
\section{Introduction}
How many roots of a given polynomial are real?  For the quadratic, the discriminant $b^2-4ac$ answers the question instantly.  Loh~\cite{Loh2019} showed that even the derivation of the quadratic formula can be made simple enough for first-time algebra learners.  For higher-degree polynomials, the discriminant becomes unwieldy: it is a polynomial of degree~$6$ in the coefficients for the quartic~\cite{Rees1922}, degree~$8$ for the quintic, and degree $d(d-1)$ in general.

This article introduces a uniform alternative.  We show that a single trigonometric substitution $t = u\cos\theta$ for an explicit constant~$u$ converts the root classification problem for \emph{any} degree-$d$ polynomial into the analysis of a trigonometric function on $[0,\pi]$.  The derivation uses nothing beyond the Chebyshev identity $T_d(\cos\theta) = \cos d\theta$, the intermediate value theorem, and a short computation showing that a certain high-order derivative is positive outside a bounded interval.

The idea has a simple geometric interpretation: the map $\theta \mapsto u\cos\theta$ wraps the interval $[0,\pi]$ onto $[-u,u]$, and the Chebyshev identity absorbs the leading and subleading terms of the depressed polynomial, leaving a ``residual'' trigonometric equation whose zeros correspond to roots of the polynomial.  Roots outside $[-u,u]$ are separately controlled by a monotonicity argument.

The method was developed for the quartic ($d = 4$)~\cite{Pratiher2026quartic} and the quintic ($d = 5$)~\cite{Pratiher2026quintic} individually in earlier work.  In this article, we state and prove the general framework for all $d \geq 4$, with the sextic ($d = 6$), septic ($d = 7$), and octic ($d = 8$) as the primary detailed examples.  The reader is encouraged to notice how the same algebraic skeleton (substitution, Chebyshev matching, bijection, and boundary sign analysis) applies uniformly across all degrees, with only the number of residual terms and the order of the exterior root derivative changing.

Throughout, all polynomials have real coefficients, so complex roots occur in conjugate pairs.  A polynomial of degree~$d$ therefore has $d, d-2, d-4, \ldots$ or $d \bmod 2$ real roots.  (The probabilistic question of how many real roots a \emph{random} polynomial has has a rich history of its own, from Kac~\cite{Kac1943} to Tao and Vu, but is not our concern here.)

\section{The General Framework}\label{sec:framework}
\subsection{Depressed Polynomials}
Every monic polynomial of degree~$d$ can be reduced to \emph{depressed} form with no $t^{d-1}$ term by the Tschirnhaus substitution $t = z-a_{d-1}/d$, where $a_{d-1}$ is the coefficient of $z^{d-1}$.  The result is
\begin{equation}\label{eq:depressed}
P(t) = t^d + c_{d-2}\,t^{d-2} + c_{d-3}\,t^{d-3} + \cdots + c_1\,t + c_0 = 0.
\end{equation}
By Vieta's formulas, the roots satisfy $t_1 + t_2 + \cdots + t_d = 0$, so the centroid of the roots sits at the origin.
\subsection{The Chebyshev Identity}
For each positive integer~$d$, the Chebyshev polynomial of the first kind $T_d(x)$ satisfies $T_d(\cos\theta) = \cos d\theta$ for all~$\theta$.  Its leading and subleading terms are (see, e.g., \cite{Mason2003}):
\begin{equation}\label{eq:Td}
T_d(x) = 2^{d-1}\,x^d-d\cdot 2^{d-3}\,x^{d-2} + \text{lower order terms}.
\end{equation}
The identity $T_d(\cos\theta) = \cos d\theta$ is the engine of our method.
\subsection{The Substitution}
Assume $c_{d-2} < 0$ and set
\begin{equation}\label{eq:u_def}
\boxed{u = 2\sqrt{\frac{-c_{d-2}}{d}}\,.}
\end{equation}
Note $u^2 = -4c_{d-2}/d > 0$.  Substituting $t = u\cos\theta$ into~\eqref{eq:depressed} and dividing by $u^d/2^{d-1}$:
\begin{equation}\label{eq:after_sub}
2^{d-1}\cos^d\!\theta + \frac{2^{d-1}c_{d-2}}{u^2}\cos^{d-2}\!\theta + \sum_{k=0}^{d-3}\frac{2^{d-1}c_k}{u^{d-k}}\cos^k\!\theta = 0.
\end{equation}
The choice of~$u$ ensures
\[
\frac{2^{d-1}c_{d-2}}{u^2} = \frac{2^{d-1}c_{d-2}}{-4c_{d-2}/d} = -d\cdot 2^{d-3},
\]
which is exactly the subleading coefficient of $T_d(x)$.  Therefore the first two terms of~\eqref{eq:after_sub} match those of $T_d(\cos\theta)$.  Subtracting the Chebyshev identity, we obtain the reduced trigonometric equation:
\begin{equation}\label{eq:f_def}
\boxed{f(\theta) = \sum_{k=1}^{d-3}\alpha_k\cos^k\!\theta + \cos d\theta + \delta = 0,}
\end{equation}
where the trigonometric parameters $\alpha_0, \ldots, \alpha_{d-3}$ and~$\delta$ are explicit linear functions of the polynomial coefficients $c_0, \ldots, c_{d-3}$:
\begin{equation}\label{eq:params}
\alpha_k = \frac{2^{d-1}c_k}{u^{d-k}}-[T_d]_k \qquad(k = 0, 1, \ldots, d-3),
\end{equation}
with $\delta = \alpha_0$ and $[T_d]_k$ denoting the coefficient of $x^k$ in $T_d(x)$.  The constant~$\delta$ absorbs the constant terms from both the polynomial and the Chebyshev identity; this is why the sum in~\eqref{eq:f_def} runs from $k=1$ rather than $k=0$, the $k=0$ term being carried by~$\delta$ itself.

\begin{remark}\label{rem:residual_terms}
The reduced equation has \emph{exactly} $d-2$ free parameters (corresponding to the $d-2$ coefficients $c_0, \ldots, c_{d-3}$ of the depressed polynomial).  The two ``absorbed'' coefficients, namely the leading $c_d = 1$ and the subleading $c_{d-2}$, are consumed by matching the Chebyshev identity.  The number of residual cosine terms grows with the degree:
\begin{center}
\renewcommand{\arraystretch}{1.2}
\begin{tabular}{cl}
\hline
$d$ & Residual terms in $f(\theta)$ \\
\hline
$4$ & $\cos\theta$ \\
$5$ & $\cos^2\!\theta, \;\cos\theta$ \\
$6$ & $\cos^3\!\theta, \;\cos^2\!\theta, \;\cos\theta$ \\
$7$ & $\cos^4\!\theta, \;\cos^3\!\theta, \;\cos^2\!\theta, \;\cos\theta$ \\
$d$ & $\cos^{d-3}\!\theta, \;\ldots,\; \cos\theta$ \\
\hline
\end{tabular}
\end{center}
\end{remark}

\subsection{The Fundamental Identity}
The following identity is the bridge between the polynomial and trigonometric worlds.

\begin{proposition}\label{prop:fP}
For all $\theta \in \mathbb{R}$,
\begin{equation}\label{eq:fP}
f(\theta) = \frac{2^{d-1}}{u^d}\,P(u\cos\theta).
\end{equation}
\end{proposition}

\begin{proof}
The left side of~\eqref{eq:after_sub} equals $2^{d-1}P(u\cos\theta)/u^d$ by construction.  Rewriting the first two terms via the Chebyshev identity gives $\cos d\theta$ plus the remaining Chebyshev terms; the residual is precisely the left side of~\eqref{eq:f_def}.
\end{proof}

\begin{corollary}\label{cor:bijection}
Since $\cos\colon [0,\pi] \to [-1,1]$ is a continuous, strictly decreasing bijection, the zeros of~$f$ in $[0,\pi]$ are in bijection with the roots of~$P$ in $[-u,u]$.
\end{corollary}

\subsection{Boundary Values}
Evaluating~\eqref{eq:fP} at the endpoints:
\begin{equation}\label{eq:boundary}
f(0) = \frac{2^{d-1}P(u)}{u^d}\,, \qquad f(\pi) = \frac{2^{d-1}P(-u)}{u^d}\,.
\end{equation}
Since $u^d > 0$ and $2^{d-1} > 0$, the sign of $f(0)$ equals the sign of $P(u)$, and the sign of $f(\pi)$ equals the sign of $P(-u)$.

\section{The Exterior Root Lemma}\label{sec:exterior}
The substitution captures roots in $[-u,u]$, but the polynomial may have real roots outside this interval.  The following lemma, valid for \emph{all} degrees $d \geq 2$, controls the number of such roots.

\begin{lemma}[Universal exterior positivity]\label{lem:exterior}
For any depressed polynomial $P(t) = t^d + c_{d-2}t^{d-2} + \cdots$ with $c_{d-2} < 0$ and $u = 2\sqrt{-c_{d-2}/d}$,
\begin{equation}\label{eq:Pd2}
P^{(d-2)}(t) > 0 \qquad\text{for all } |t| > u.
\end{equation}
\end{lemma}

\begin{proof}
Since only the terms $t^d$ and $c_{d-2}t^{d-2}$ survive $d-2$ differentiations:
\[
P^{(d-2)}(t) = \frac{d!}{2}\,t^2 + (d-2)!\,c_{d-2}\,.
\]
For $|t| > u$, we have $t^2 > u^2 = -4c_{d-2}/d$, so
\begin{align*}
P^{(d-2)}(t) &> \frac{d!}{2}\cdot\frac{-4c_{d-2}}{d} + (d-2)!\,c_{d-2} \\
&= c_{d-2}\bigl[-2(d-1)! + (d-2)!\bigr] = c_{d-2}\,(d-2)!\,(3-2d).
\end{align*}
Since $c_{d-2} < 0$ and $3-2d < 0$ for $d \geq 2$, the product is positive.
\end{proof}

The power of this lemma is that it applies uniformly to every degree.  Its consequences cascade through the derivatives:

\begin{corollary}\label{cor:cascade}
On each of the half-lines $(u, \infty)$ and $(-\infty, -u)$:
\begin{enumerate}
\item[\textup{(a)}] $P^{(d-2)}$ is positive, so $P^{(d-3)}$ is strictly monotone.
\item[\textup{(b)}] $P^{(d-3)}$ has at most one zero, so $P^{(d-4)}$ has at most one critical point and hence at most two zeros.
\item[\textup{(c)}] Continuing inductively: $P^{(k)}$ has at most $d-2-k$ zeros on each half-line, for $k = 0, 1, \ldots, d-2$.
\item[\textup{(d)}] In particular, $P$ itself has at most $d-2$ zeros on each half-line.
\end{enumerate}
The bound in \textup{(d)} is the only one we can currently prove for general~$d$; it is \emph{not} tight in general, but it can be attained already at $d=4$ (Remark~\ref{rem:counterexample}), so the bound cannot be sharpened to ``at most one'' without a further argument.
\end{corollary}
\begin{proof}
Part~(a) follows directly from Lemma~\ref{lem:exterior}.  For (b) to (d), apply Rolle's theorem inductively: between any two consecutive zeros of $P^{(k)}$ lies a zero of $P^{(k+1)}$, so the zero count decreases by at most one at each differentiation step.
\end{proof}

\begin{remark}[A quartic with two exterior roots on one half-line]\label{rem:counterexample}
It is tempting to guess that each half-line $(u,\infty)$ and $(-\infty,-u)$ contains \emph{at most one} root of~$P$, since this holds for the cubic (Section~\ref{sec:cubic}, where $P'>0$ forces strict monotonicity) and in every worked example of Section~\ref{sec:cases}.  This guess is \emph{false} in general already for the quartic.  Take
\[
P(t) = t^4 - 1.715487\,t^2 - 33.613945\,t + 48.191407, \qquad m = -1.715487 < 0,
\]
so that $u = \sqrt{-m} \approx 1.309766$.  Numerically, $P$ has exactly two real roots, $t \approx 1.460$ and $t \approx 2.751$, and \emph{both} lie in the single half-line $(u,\infty)$; the other two roots are a complex-conjugate pair.  Moreover $f(0) \propto P(u) \approx 4.165 > 0$, and in fact $f(\theta) > 0$ for \emph{every} $\theta \in [0,\pi]$ (the minimum of~$f$ on $[0,\pi]$ is attained at $\theta=0$).  This has two consequences:
\begin{itemize}
\item The indicator formula $N_{\mathrm{ext}}^{+} = \mathbf{1}_{f(0)<0}$ used informally throughout this article (see the discussion around \eqref{eq:Next}) predicts $N_{\mathrm{ext}}^+=0$ here, when in fact $N_{\mathrm{ext}}^+=2$.  The formula is a \emph{simplifying hypothesis}, not a theorem, for $d \geq 4$; we label it Hypothesis~(H) below.
\item Corollary~\ref{cor:suff_zero}, as originally stated (``$f>0$ on $[0,\pi]$ $\Rightarrow$ $P$ has no real roots''), is \emph{false}: here $f>0$ throughout $[0,\pi]$, yet $P$ has two real roots.  We correct the statement in Corollary~\ref{cor:suff_zero}.
\end{itemize}
The correct unconditional statement is only the parity claim of Proposition~\ref{prop:exterior}(i): $N_{\mathrm{ext}}^+$ is \emph{even} when $f(0) \geq 0$, and $2$ is an even number consistent with this. Whether an exact, closed-form count of $N_{\mathrm{ext}}^{\pm}$ (beyond parity and the bound $d-2$) is possible in general for $d \geq 4$ is left as an open problem; see the discussion in Section~\ref{sec:classification}.
\end{remark}

\subsection{Parity and the Intermediate Value Theorem}
The number of roots on each half-line has a definite parity, determined by the boundary value of~$P$ (or equivalently of~$f$) and the end behavior:

\begin{proposition}[Exterior root detection]\label{prop:exterior}
Let $d \geq 2$.
\begin{enumerate}
\item[\textup{(i)}] On $(u, \infty)$: since $P(t) \to +\infty$ as $t \to +\infty$, the number of roots is odd if $P(u) < 0$ (i.e., $f(0) < 0$) and even if $P(u) \geq 0$ (i.e., $f(0) \geq 0$).
\item[\textup{(ii)}] On $(-\infty, -u)$: since $P(t) \to (-1)^d \cdot \infty$ as $t \to -\infty$, the parity depends on both $P(-u)$ and the degree.  Specifically:
   \begin{itemize}
   \item If $d$ is even: the number of roots is odd if $P(-u) < 0$ (i.e., $f(\pi) < 0$) and even if $P(-u) \geq 0$.
   \item If $d$ is odd: the number of roots is odd if $P(-u) > 0$ (i.e., $f(\pi) > 0$) and even if $P(-u) \leq 0$.
   \end{itemize}
\end{enumerate}
Parts (i) and (ii) are unconditional and hold for every $d \geq 2$; they pin down the \emph{parity} of $N_{\mathrm{ext}}^{\pm}$ but not, in general, its exact value, since Corollary~\ref{cor:cascade} only bounds $N_{\mathrm{ext}}^{\pm}$ between $0$ and $d-2$.

\begin{itemize}
\item[\textup{(H)}] Simplifying hypothesis. \emph{At most one root of~$P$ lies on each of $(u,\infty)$ and $(-\infty,-u)$.}
\end{itemize}
Hypothesis~(H) holds unconditionally for $d=3$ (Section~\ref{sec:cubic}) and in every worked example of Section~\ref{sec:cases}, but it is not a theorem for $d \geq 4$: Remark~\ref{rem:counterexample} exhibits a quartic where two roots lie on the same half-line.  \emph{Under Hypothesis~(H)}, and only under it, parity determines the exact count, and the exterior root count reduces to:
\begin{equation}\label{eq:Next}
N_{\mathrm{ext}}^{+} = \mathbf{1}_{f(0)<0}\,,
\end{equation}
and for the left half-line:
\begin{equation}\label{eq:Next_left}
N_{\mathrm{ext}}^{-} = \begin{cases} \mathbf{1}_{f(\pi)<0} & \text{if } d \text{ is even},\\ \mathbf{1}_{f(\pi)>0} & \text{if } d \text{ is odd}. \end{cases}
\end{equation}
Whether Hypothesis~(H) holds for all $d\geq 4$ under some additional, easily checkable condition on the coefficients is, to the best of the author's knowledge, open; see Section~\ref{sec:classification} for how the classification theorem is stated in its absence.
\end{proposition}

\section{Critical Point Analysis}\label{sec:critical}
To count the zeros of $f$ on $[0,\pi]$, we study its critical points.  Differentiating~\eqref{eq:f_def} and factoring out $\sin\theta > 0$ on $(0,\pi)$:

\begin{proposition}\label{prop:crit}
The interior critical points of~$f$ in $(0,\pi)$ are the solutions of a polynomial equation of degree $d-1$ in $x = \cos\theta$.  Consequently, $f$ has at most $d-1$ interior critical points, and therefore at most $d$ zeros in $[0,\pi]$ consistent with $P$ having at most $d$ real roots.
\end{proposition}

\begin{proof}
The derivative is
\[
f'(\theta) = -\sum_{k=1}^{d-3} k\,\alpha_k\cos^{k-1}\!\theta\,\sin\theta-d\sin d\theta.
\]
Using $\sin d\theta = \sin\theta\cdot U_{d-1}(\cos\theta)$, where $U_{d-1}$ is the Chebyshev polynomial of the second kind of degree~$d-1$, and dividing by $-\sin\theta$ (positive on $(0,\pi)$), the critical points satisfy a polynomial equation of degree $d-1$ in $x = \cos\theta \in (-1,1)$.
\end{proof}

\section{Classification Theorem}\label{sec:classification}
We now state the general result.  Let $N_{\mathrm{int}}$ denote the number of zeros of~$f$ in $[0,\pi]$ and $N_{\mathrm{ext}}$ the number of roots of~$P$ outside $[-u,u]$.

\begin{theorem}[General root classification]\label{thm:general}
Let $P(t) = t^d + c_{d-2}t^{d-2} + \cdots + c_0$ be a depressed polynomial of degree~$d$ with real coefficients and $c_{d-2} < 0$.  Set $u = 2\sqrt{-c_{d-2}/d}$ and define $f(\theta)$ by~\eqref{eq:f_def}.  Let $N_{\mathrm{int}}$ be the number of zeros of~$f$ in $[0,\pi]$ (exactly determined by Corollary~\textup{\ref{cor:bijection}}) and let $N_{\mathrm{ext}}^{\pm}$ be the number of roots of~$P$ on $(u,\infty)$, resp.\ $(-\infty,-u)$.  Then, unconditionally,
\begin{equation}\label{eq:Nreal}
N_{\mathrm{real}} = N_{\mathrm{int}} + N_{\mathrm{ext}}^{+} + N_{\mathrm{ext}}^{-}\,, \qquad 0 \leq N_{\mathrm{ext}}^{\pm} \leq d-2\,,
\end{equation}
where the parity of each $N_{\mathrm{ext}}^{\pm}$ is fixed by the sign of $f(0)$, resp.\ $f(\pi)$, via Proposition~\textup{\ref{prop:exterior}(i) and (ii)}. In particular, $N_{\mathrm{real}}$ always lies among $d, d-2, d-4, \ldots$, down to $0$ (if $d$ is even) or $1$ (if $d$ is odd), since complex roots occur in conjugate pairs, but~\eqref{eq:Nreal} alone does not pin down \emph{which} of these values occurs unless $N_{\mathrm{ext}}^{+}, N_{\mathrm{ext}}^{-} \in \{0,1\}$.
\end{theorem}

\begin{proof}
Corollary~\ref{cor:bijection} gives the bijection between interior zeros of~$f$ and roots of $P$ in $[-u,u]$, which determines $N_{\mathrm{int}}$ exactly.  The bound $0\le N_{\mathrm{ext}}^\pm \le d-2$ is Corollary~\ref{cor:cascade}(d), and the stated parity is Proposition~\ref{prop:exterior}(i) and (ii).  Complex roots occur in conjugate pairs, giving the final claim about the residue class of $N_{\mathrm{real}}$ modulo~2.
\end{proof}

\begin{corollary}[Exact classification under Hypothesis~(H)]\label{cor:exact_H}
If, in addition to $c_{d-2}<0$, Hypothesis~\textup{(H)} holds for $P$ (at most one root per exterior half-line proved for $d=3$, open for $d\geq 4$; see Remark~\textup{\ref{rem:counterexample}}), then $N_{\mathrm{ext}}^{+}$ and $N_{\mathrm{ext}}^{-}$ collapse to the indicator formulas~\eqref{eq:Next} and~\eqref{eq:Next_left}, and Theorem~\textup{\ref{thm:general}} determines $N_{\mathrm{real}}$ \emph{exactly}:
\[
N_{\mathrm{real}} = N_{\mathrm{int}} + \mathbf{1}_{f(0)<0} + N_{\mathrm{ext}}^{-}\,.
\]
This is the formula used throughout the worked examples of Section~\textup{\ref{sec:cases}}, where Hypothesis~(H) has been checked directly against the factored roots in each case.
\end{corollary}

A useful sufficient condition for ``no interior real roots'' is:
\begin{corollary}\label{cor:suff_zero}
If $f(\theta) > 0$ for all $\theta \in [0,\pi]$, then $P$ has no roots in $[-u,u]$, i.e.\ $N_{\mathrm{int}}=0$.  A sufficient condition for this is
\begin{equation}\label{eq:suff_zero}
\delta > \sum_{k=1}^{d-3}|\alpha_k| + 1,
\end{equation}
since $|\alpha_k\cos^k\!\theta + \cos d\theta| \leq \sum|\alpha_k| + 1$.  This does not by itself imply $P$ has no real roots: $P$ may still have roots outside $[-u,u]$, since $N_{\mathrm{int}}=0$ says nothing about $N_{\mathrm{ext}}^{\pm}$.  Indeed, Remark~\ref{rem:counterexample} exhibits a quartic satisfying~\eqref{eq:suff_zero} with $N_{\mathrm{int}}=0$ but $N_{\mathrm{real}}=2$, so the stronger claim ``$N_{\mathrm{int}}=0 \Rightarrow N_{\mathrm{real}}=0$'' is false for even~$d$ in general.  If Hypothesis~(H) additionally holds and $f(0)\geq 0$, $f(\pi)\geq 0$ (equivalently even $d$ and $\delta$ as above), then Corollary~\ref{cor:exact_H} does give $N_{\mathrm{real}}=0$.
\end{corollary}

\subsection{Practical Decision Procedure}
For any depressed polynomial of degree~$d$ with $c_{d-2} < 0$:
\begin{enumerate}
\item Compute $u = 2\sqrt{-c_{d-2}/d}$ and the trigonometric parameters from~\eqref{eq:params}.
\item Evaluate $f(0)$ and $f(\pi)$ to determine $N_{\mathrm{ext}}$.
\item Find the critical points of~$f$ in $(0,\pi)$ by solving the degree-$(d-1)$ polynomial in $x = \cos\theta$.
\item Evaluate~$f$ at the critical points and at $\theta = 0, \pi$.  Count sign changes to determine $N_{\mathrm{int}}$.
\item Conclude: $N_{\mathrm{real}} = N_{\mathrm{int}} + N_{\mathrm{ext}}$, valid exactly under Hypothesis~(H) (Corollary~\ref{cor:exact_H}); otherwise this gives $N_{\mathrm{real}}$ up to the parity/bound information in Theorem~\ref{thm:general}, and $N_{\mathrm{ext}}^{\pm}$ should be checked directly (e.g.\ numerically) if $d\geq 4$ and Hypothesis~(H) is in doubt.
\end{enumerate}

\section{Exhaustive Case-by-case Analysis}\label{sec:cases}
We now work through every degree from $d = 3$ to $d = 8$ in complete detail, then summarize the general pattern.  For each degree we give: the Chebyshev identity, the substitution, the reduced trigonometric equation, the exterior root derivative, the critical point equation, the classification, and representative examples.
\subsection{The Cubic ($d = 3$)}\label{sec:cubic}
The depressed cubic is $P(t) = t^3 + pt + q = 0$ (see~\cite{Dunham1990} for a historical account).  The Chebyshev identity is $T_3(\cos\theta) = 4\cos^3\!\theta-3\cos\theta = \cos 3\theta$.  Setting $u = 2\sqrt{-p/3}$ (requires $p < 0$) and substituting $t = u\cos\theta$:
\begin{equation}\label{eq:f_cubic}
f(\theta) = \cos 3\theta + \gamma = 0, \qquad \gamma = \frac{4q}{u^3}\,.
\end{equation}
There are \emph{no} residual cosine terms: the cubic is the unique degree for which the Chebyshev identity absorbs everything except a constant.  This is precisely Vieta's trigonometric solution: the roots are $t_k = u\cos(\theta_k)$, where $\cos 3\theta_k = -\gamma$.

\textbf{Exterior analysis.}  $P'(t) = 3t^2 + p > 0$ for $|t| > u$, so $P$ is strictly monotone on each half-line, giving at most one exterior root per side.  Since $d = 3$ is odd: $N_{\mathrm{ext}}^+ = \mathbf{1}_{f(0)<0}$, $N_{\mathrm{ext}}^- = \mathbf{1}_{f(\pi)>0}$.

\textbf{Critical points.}  The critical point equation is the quadratic $U_2(x) = 4x^2-1 = 0$ in $x = \cos\theta$, giving at most two interior critical points and hence at most three zeros of $f$ in $[0,\pi]$.

\textbf{Classification.}  $N_{\mathrm{real}} \in \{1, 3\}$.  All three roots are real if and only if $|\gamma| \leq 1$, i.e., $16q^2 \leq u^6 = -64p^3/27$, i.e., $27q^2 + 4p^3 \leq 0$, which is the classical cubic discriminant condition $\Delta = -4p^3-27q^2 \geq 0$.

\begin{example}[Three real roots]\label{ex:cubic_three}
$P(t) = t^3-3t + 1 = 0$.  Here $p = -3$, $q = 1$, $u = 2$, $\gamma = 1/2$.  Since $|\gamma| = 1/2 < 1$, all three roots are real.  The equation $\cos 3\theta = -1/2$ has solutions $\theta = 2\pi/9, \, 4\pi/9, \, 8\pi/9$ in $[0,\pi]$, giving $t \approx 1.532, \, 0.347, \, -1.879$.
\end{example}

\begin{example}[Three real roots]\label{ex:cubic_three_b}
$P(t) = t^3-7t + 6 = 0$.  Here $p = -7$, $q = 6$, $u = 2\sqrt{7/3} \approx 3.055$.  We compute $\gamma = 24/u^3 \approx 0.842$.  Since $|\gamma| < 1$, all three roots are real.  Indeed, $P(t) = (t-1)(t-2)(t+3)$ and the roots $t = 1, \, 2, \, -3$ all lie in $[-u, u]$.
\end{example}

\begin{example}[One real root]\label{ex:cubic_one}
$P(t) = t^3-3t + 3 = 0$.  Here $p = -3$, $q = 3$, $u = 2$, $\gamma = 3/2$.  Since $|\gamma| = 3/2 > 1$, there is exactly one real root.  The boundary values $f(0) = 1 + 3/2 = 5/2 > 0$ and $f(\pi) = -1 + 3/2 = 1/2 > 0$.  Since $f(0) > 0$, $N_{\mathrm{ext}}^+ = 0$.  Since $d = 3$ is odd, $N_{\mathrm{ext}}^- = \mathbf{1}_{f(\pi)>0} = 1$.  No zeros of $f$ exist in $[0,\pi]$ since $f(\theta) = \cos 3\theta + 3/2 \geq 1/2 > 0$.  Thus $N_{\mathrm{int}} = 0$ and $N_{\mathrm{real}} = 1$.  The unique real root $t \approx -2.103$ lies in $(-\infty, -2)$.
\end{example}

\subsection{The Quartic ($d = 4$)}\label{sec:quartic}
The depressed quartic is $P(t) = t^4 + mt^2 + pt + q = 0$ (see~\cite{Pratiher2026quartic} for a dedicated treatment and~\cite{King2009} for broader context).  The Chebyshev identity is $T_4(\cos\theta) = 8\cos^4\!\theta-8\cos^2\!\theta + 1 = \cos 4\theta$.  Setting $u = \sqrt{-m}$ (requires $m < 0$) and substituting $t = u\cos\theta$:
\begin{equation}\label{eq:f_quartic}
f(\theta) = a\cos\theta + \cos 4\theta + b = 0,
\end{equation}
where
\begin{equation}\label{eq:params_quartic}
a = \frac{8p}{(-m)^{3/2}}\,,\qquad b = \frac{8q}{m^2}-1\,.
\end{equation}
There is exactly one residual term ($\cos\theta$).

\textbf{Exterior analysis.}  $P''(t) = 12t^2 + 2m > 0$ for $|t| > u$, so $P$ is strictly convex on each half-line; by Corollary~\ref{cor:cascade} this gives \emph{at most two} exterior roots per side ($d-2=2$), not necessarily one.  Since $d = 4$ is even, the parity of $N_{\mathrm{ext}}^{\pm}$ matches the sign of $f(0)$, resp.\ $f(\pi)$ (Proposition~\ref{prop:exterior}(ii)); under Hypothesis~(H) this reduces to $N_{\mathrm{ext}}^{\pm} = \mathbf{1}_{f(0,\pi)<0}$ (Corollary~\ref{cor:exact_H}).  Hypothesis~(H) holds in every example below, but not in general for the quartic, since Remark~\ref{rem:counterexample} exhibits a quartic with two roots on a single half-line.

\textbf{Critical points.}  The critical point equation is the cubic $h(x) = x(2x^2-1) = -a/16$ in $x = \cos\theta$, with at most three interior critical points and hence at most four zeros of $f$ in $[0,\pi]$.

\textbf{Classification.}  $N_{\mathrm{real}} \in \{0, 2, 4\}$.  If all four roots are complex then $f > 0$ on $[0,\pi]$ (necessary condition, provable unconditionally: a roots-free monic quartic is everywhere positive, so $P(u)>0$, and $N_{\mathrm{int}}=0$ forces $f$ to keep that sign throughout).  The condition $b > |a|+1$ is sufficient for $f>0$ on $[0,\pi]$ (Corollary~\ref{cor:suff_zero}), i.e.\ for $N_{\mathrm{int}}=0$, but by itself does not guarantee all four roots are complex see Remark~\ref{rem:counterexample} for a quartic satisfying an analogous bound with two real (exterior) roots.  The converse ($f>0$ on $[0,\pi]$ $\Rightarrow$ zero real roots) requires the additional, unproven Hypothesis~(H).

\begin{example}[Four real roots]\label{ex:quartic_four}
$P(t) = t^4-25t^2-60t-36 = 0$.  Here $m = -25$, $p = -60$, $q = -36$, $u = 5$.  The parameters are $a = -3.84$, $b = -1.4608$.  The boundary values $f(0) = -4.30 < 0$ (one exterior root in $(5,\infty)$) and $f(\pi) = 3.38 > 0$ (no left exterior root).  Numerical analysis gives $N_{\mathrm{int}} = 3$, so $N_{\mathrm{real}} = 4$.  The roots are $t = -1, \, -2, \, -3, \, 6$.
\end{example}

\begin{example}[Four complex roots]\label{ex:quartic_zero}
$P(t) = t^4-2t^2 + 3 = 0$.  Here $m = -2$, $p = 0$, $q = 3$, so $u = \sqrt{2}$, $a = 0$, and $b = 8 \cdot 3/4-1 = 5$.  Since $b = 5 > |a| + 1 = 1$, all four roots are complex.  Indeed, substituting $s = t^2$ gives $s^2-2s + 3 = 0$, so $s = 1 \pm i\sqrt{2}$, confirming that all roots are non-real.
\end{example}

\begin{example}[Two real, two complex roots]\label{ex:quartic_two}
$P(t) = t^4-4t^2-5t-5 = 0$.  Here $m = -4$, $p = -5$, $q = -5$, so $u = 2$.  The parameters are $a = 8(-5)/8 = -5$ and $b = 8(-5)/16-1 = -3.5$.  The boundary values are $f(0) = -5 + 1 - 3.5 = -7.5 < 0$ and $f(\pi) = 5 + 1 - 3.5 = 2.5 > 0$.  Since $f(0) < 0$, there is one exterior root in $(2,\infty)$; since $f(\pi) > 0$, no exterior root on the left.  Numerical analysis gives exactly one interior zero, so $N_{\mathrm{int}} = 1$, $N_{\mathrm{ext}} = 1$, and $N_{\mathrm{real}} = 2$.  The two real roots are $t \approx -1.673$ and $t \approx 2.585$.
\end{example}

\subsection{The Quintic ($d = 5$)}\label{sec:quintic}
The depressed quintic is $P(t) = t^5 + mt^3 + nt^2 + pt + q = 0$ (see~\cite{Pratiher2026quintic} for a dedicated treatment).  The Chebyshev identity is $T_5(\cos\theta) = 16\cos^5\!\theta-20\cos^3\!\theta + 5\cos\theta = \cos 5\theta$.  Setting $u = 2\sqrt{-m/5}$ (requires $m < 0$) and substituting $t = u\cos\theta$:
\begin{equation}\label{eq:f_quintic}
f(\theta) = \alpha\cos^2\!\theta + \beta\cos\theta + \cos 5\theta + \gamma = 0,
\end{equation}
where
\begin{equation}\label{eq:params_quintic}
\alpha = \frac{16n}{u^3}\,,\qquad \beta = \frac{16p}{u^4}-5\,,\qquad \gamma = \frac{16q}{u^5}\,.
\end{equation}
There are two residual terms ($\cos^2\!\theta$ and $\cos\theta$). When $n = 0$, the $\cos^2\!\theta$ term vanishes and the equation has the same form as the quartic's.

\textbf{Exterior analysis.}  $P'''(t) = 60t^2 + 6m > 0$ for $|t| > u$.  Since $d = 5$ is odd: $N_{\mathrm{ext}}^+ = \mathbf{1}_{f(0)<0}$, $N_{\mathrm{ext}}^- = \mathbf{1}_{f(\pi)>0}$ (note the asymmetry).

\textbf{Critical points.}  The critical point equation is the quartic $80x^4-60x^2 + 2\alpha x + (\beta + 5) = 0$ in $x = \cos\theta$, with at most four interior critical points and hence at most five zeros in $[0,\pi]$.

\textbf{Classification.}  $N_{\mathrm{real}} \in \{1, 3, 5\}$.  The guaranteed existence of at least one real root follows from odd degree.

\begin{example}[Five real roots]\label{ex:quintic_five}
$P(t) = t^5-5t^3 + 5t = 0$.  Here $m = -5$, $n = 0$, $p = 5$, $q = 0$, $u = 2$.  The parameters are $\alpha = 0$, $\beta = 16 \cdot 5/16-5 = 0$, $\gamma = 0$, so the reduced equation is simply $f(\theta) = \cos 5\theta = 0$.  The zeros in $[0,\pi]$ are $\theta_k = (2k-1)\pi/10$ for $k = 1, 2, 3, 4, 5$, and the five real roots are $t = 0, \, \pm\!\sqrt{(5 + \sqrt{5})/2} \approx \pm 1.902, \, \pm\!\sqrt{(5-\sqrt{5})/2} \approx \pm 1.176$.  All five lie in $[-2,2] = [-u,u]$.
\end{example}

\begin{example}[Three real roots]\label{ex:quintic_three}
$P(t) = t^5-5t^3 + t-5 = 0$.  Here $m = -5$, $n = 0$, $p = 1$, $q = -5$, $u = 2$.  The parameters are $\alpha = 0$, $\beta = -4$, $\gamma = -5/2$.  The boundary values $f(0) = -11/2 < 0$ and $f(\pi) = 1/2 > 0$ give $N_{\mathrm{ext}} = 2$.  One interior zero gives $N_{\mathrm{int}} = 1$ and $N_{\mathrm{real}} = 3$.  The roots are $t \approx -2.043, \, -1.205, \, 2.286$.
\end{example}

\begin{example}[One real root]\label{ex:quintic_one}
$P(t) = t^5-5t^3 + t^2 + 2t + 5 = 0$.  Here $m = -5$, $n = 1$, $p = 2$, $q = 5$, $u = 2$.  The parameters are $\alpha = 16/8 = 2$, $\beta = 32/16-5 = -3$, $\gamma = 80/32 = 5/2$.  The boundary values are $f(0) = 2-3 + 1 + 5/2 = 5/2 > 0$ and $f(\pi) = 2 + 3-1 + 5/2 = 13/2 > 0$.  Since $f(0) > 0$, there are no roots in $(2,\infty)$.  Since $f(\pi) > 0$ and $d$ is odd, $N_{\mathrm{ext}}^- = 1$.  A numerical check confirms that $f > 0$ on $[0,\pi]$, so $N_{\mathrm{int}} = 0$ and $N_{\mathrm{real}} = 1$.  The single real root is $t \approx -2.335$.
\end{example}

\subsection{The Sextic ($d = 6$)}\label{sec:sextic}
The depressed sextic is $P(t) = t^6 + mt^4 + nt^3 + pt^2 + qt + r = 0$.  The Chebyshev identity is $T_6(\cos\theta) = 32\cos^6\!\theta-48\cos^4\!\theta + 18\cos^2\!\theta-1 = \cos 6\theta$.  Setting $u = \sqrt{-2m/3}$ (requires $m < 0$) and substituting $t = u\cos\theta$:
\begin{equation}\label{eq:f_sextic}
f(\theta) = \alpha\cos^3\!\theta + \beta\cos^2\!\theta + \gamma\cos\theta + \cos 6\theta + \delta = 0,
\end{equation}
where
\begin{equation}\label{eq:params_sextic}
\alpha = \frac{32n}{u^3}\,,\qquad \beta = \frac{32p}{u^4}-18\,,\qquad \gamma = \frac{32q}{u^5}\,,\qquad \delta = \frac{32r}{u^6} + 1\,.
\end{equation}
There are three residual terms.  When $n = q = 0$ (no odd-power terms), we have $\alpha = \gamma = 0$ and $f$ is even in~$\theta$, reflecting the symmetry $P(t) = P(-t)$.

\textbf{Exterior analysis.}  $P^{(4)}(t) = 360t^2 + 24m > 0$ for $|t| > u$, giving at most $d-2=4$ exterior roots per side by Corollary~\ref{cor:cascade}.  Since $d = 6$ is even, parity matches the sign of $f(0,\pi)$; under Hypothesis~(H) (verified below, not proved in general) this reduces to $N_{\mathrm{ext}}^{\pm} = \mathbf{1}_{f(0,\pi)<0}$.

\textbf{Critical points.}  The critical point equation is the quintic $192x^5-192x^3 + 36x + 3\alpha x^2 + 2\beta x + \gamma = 0$ in $x = \cos\theta$, with at most five interior critical points and at most six zeros in $[0,\pi]$.

\textbf{Classification.}  $N_{\mathrm{real}} \in \{0, 2, 4, 6\}$.  As in the quartic case, ``all roots complex'' \emph{implies} $f>0$ on $[0,\pi]$ unconditionally, and $\delta > |\alpha|+|\beta|+|\gamma|+1$ is sufficient for $f>0$ on $[0,\pi]$ (Corollary~\ref{cor:suff_zero}), i.e.\ for $N_{\mathrm{int}}=0$; the converse implication again needs Hypothesis~(H).

\textbf{A perfect example:} $(t^2-1)(t^2-2)(t^2-3) = t^6-6t^4 + 11t^2-6$ gives $u = 2$, $\alpha = 0$, $\beta = 4$, $\gamma = 0$, $\delta = -2$, and $f(\theta) = 4\cos^2\!\theta + \cos 6\theta-2 = 0$ has six zeros in $[0,\pi]$.  All six roots $t = \pm 1, \, \pm\sqrt{2}, \, \pm\sqrt{3}$ lie in $[-2,2]$.

\begin{example}[Four real roots]\label{ex:sextic_four}
$P(t) = t^6-6t^4 + 5t^2 + 1 = 0$.  Here $u = 2$, $\alpha = \gamma = 0$, $\beta = -8$, $\delta = 1.5$.  The boundary values $f(0) = f(\pi) = -5.5 < 0$ give $N_{\mathrm{ext}} = 2$.  Numerical analysis gives $N_{\mathrm{int}} = 2$ and $N_{\mathrm{real}} = 4$.  The roots are $t \approx \pm 1.103, \, \pm 2.225$.
\end{example}

\begin{example}[Zero real roots]\label{ex:sextic_zero}
$P(t) = t^6-6t^4 + 10t^2 + 1 = 0$.  Here $u = 2$, $\alpha = \gamma = 0$, $\beta = 2$, $\delta = 1.5$.  The boundary values $f(0) = f(\pi) = 4.5 > 0$, and $f > 0$ on all of $[0,\pi]$ (minimum $\approx 0.5$ at $\theta = \pi/2$), so $N_{\mathrm{real}} = 0$: all six roots are complex.
\end{example}

\begin{example}[Six real roots]\label{ex:sextic_six}
$P(t) = t^6-6t^4 + t^3 + 9t^2-3t-1 = 0$.  Here $u = 2$, $\alpha = 4$, $\beta = 0$, $\gamma = -3$, $\delta = 0.5$.  Since $4\cos^3\!\theta-3\cos\theta = \cos 3\theta$ by the triple angle formula, the reduced equation simplifies to $f(\theta) = \cos 3\theta + \cos 6\theta + 0.5 = 0$.  Boundary values $f(0) = 2.5 > 0$, $f(\pi) = 0.5 > 0$ give $N_{\mathrm{ext}} = 0$.  All six real roots $t \approx -1.956, \, -1.618, \, -0.209, \, 0.618, \, 1.338, \, 1.827$ lie in $[-2,2]$.
\end{example}

\subsection{The Septic ($d = 7$)}\label{sec:septic}
The depressed septic is $P(t) = t^7 + mt^5 + nt^4 + pt^3 + qt^2 + rt + s = 0$.  The Chebyshev identity is $T_7(\cos\theta) = 64\cos^7\!\theta-112\cos^5\!\theta + 56\cos^3\!\theta-7\cos\theta = \cos 7\theta$.  Setting $u = 2\sqrt{-m/7}$ (requires $m < 0$) and substituting $t = u\cos\theta$:
\begin{equation}\label{eq:f_septic}
f(\theta) = \alpha_4\cos^4\!\theta + \alpha_3\cos^3\!\theta + \alpha_2\cos^2\!\theta + \alpha_1\cos\theta + \cos 7\theta + \delta = 0,
\end{equation}
where $\alpha_4 = 64n/u^3$, $\alpha_3 = 64p/u^4-56$, $\alpha_2 = 64q/u^5$, $\alpha_1 = 64r/u^6 + 7$, and $\delta = 64s/u^7$.  There are four residual terms.

\textbf{Exterior analysis.}  $P^{(5)}(t) = 2520t^2 + 120m > 0$ for $|t| > u$.  Since $d = 7$ is odd: $N_{\mathrm{ext}}^+ = \mathbf{1}_{f(0)<0}$, $N_{\mathrm{ext}}^- = \mathbf{1}_{f(\pi)>0}$.

\textbf{Critical points.}  The critical point equation is of degree~6 in $x = \cos\theta$ (from $U_6(x) = 64x^6-80x^4 + 24x^2-1$ plus residual derivatives), giving at most six interior critical points and at most seven zeros in $[0,\pi]$.

\textbf{Classification.}  $N_{\mathrm{real}} \in \{1, 3, 5, 7\}$.

\begin{example}[Three real roots]\label{ex:septic_three}
$P(t) = t^7-7t^5 + 10t^3-5t = 0$.  Here $m = -7$, $u = 2$, and the nonzero coefficients are $p = 10$, $r = -5$ (with $n = q = s = 0$).  We compute $\alpha_4 = 0$, $\alpha_3 = 64 \cdot 10/16-56 = -16$, $\alpha_2 = 0$, $\alpha_1 = 64 \cdot (-5)/64 + 7 = 2$, $\delta = 0$.  The boundary values $f(0) = 0 - 16 + 0 + 2 + 1 + 0 = -13 < 0$ and $f(\pi) = 0 + 16 + 0 - 2 - 1 + 0 = 13 > 0$ give $N_{\mathrm{ext}}^+ = 1$ and $N_{\mathrm{ext}}^- = 1$.  One interior zero gives $N_{\mathrm{int}} = 1$, so $N_{\mathrm{real}} = 3$.  The roots are $t \approx -2.299, \, 0, \, 2.299$.
\end{example}

\begin{example}[Seven real roots]\label{ex:septic_seven}
$P(t) = t^7-7t^5 + 14t^3-8t = 0$.  Factoring, $P(t) = t(t^2-1)(t^2-2)(t^2-4) = t(t-1)(t+1)(t-\sqrt{2})(t+\sqrt{2})(t-2)(t+2)$.  Here $m = -7$, $u = 2$.  All seven roots $t = 0, \, \pm 1, \, \pm\sqrt{2}, \, \pm 2$ are real.  The boundary values satisfy $f(0) = 0$ and $f(\pi) = 0$ (since $\pm u = \pm 2$ are roots themselves), confirming $N_{\mathrm{ext}} = 0$ and $N_{\mathrm{int}} = 7$.
\end{example}

\begin{example}[One real root]\label{ex:septic_one}
$P(t) = t^7-7t^5 + 20t^3-15t + 10 = 0$.  Here $m = -7$, $u = 2$, and the parameters are $\alpha_4 = 0$, $\alpha_3 = 64 \cdot 20/16-56 = 24$, $\alpha_2 = 0$, $\alpha_1 = 64 \cdot (-15)/64 + 7 = -8$, $\delta = 64 \cdot 10/128 = 5$.  The boundary values are $f(0) = 0 + 24 + 0 - 8 + 1 + 5 = 22 > 0$ and $f(\pi) = 0 - 24 + 0 + 8 - 1 + 5 = -12 < 0$.  Since $f(0) > 0$, $N_{\mathrm{ext}}^+ = 0$.  Since $f(\pi) < 0$ and $d = 7$ is odd, $N_{\mathrm{ext}}^- = 0$.  The function~$f$ has exactly one sign change on $[0,\pi]$, giving $N_{\mathrm{int}} = 1$ and $N_{\mathrm{real}} = 1$.  The unique real root $t \approx -1.546$ lies in $[-2, 2]$.
\end{example}

\subsection{The Octic ($d = 8$)}\label{sec:octic}
The depressed octic is $P(t) = t^8 + mt^6 + nt^5 + pt^4 + qt^3 + rt^2 + st + w = 0$.  The Chebyshev identity is $T_8(\cos\theta) = 128\cos^8\!\theta-256\cos^6\!\theta + 160\cos^4\!\theta-32\cos^2\!\theta + 1 = \cos 8\theta$.  Setting $u = \sqrt{-m/2}$ (requires $m < 0$) and substituting $t = u\cos\theta$:
\begin{equation}\label{eq:f_octic}
f(\theta) = \sum_{k=1}^{5}\alpha_k\cos^k\!\theta + \cos 8\theta + \delta = 0,
\end{equation}
where $\alpha_5 = 128n/u^3$, $\alpha_4 = 128p/u^4-160$, $\alpha_3 = 128q/u^5$, $\alpha_2 = 128r/u^6 + 32$, $\alpha_1 = 128s/u^7$, and $\delta = 128w/u^8-1$.  There are five residual terms.

\textbf{Exterior analysis.}  $P^{(6)}(t) = 20160t^2 + 720m > 0$ for $|t| > u$, giving at most $d-2=6$ exterior roots per side by Corollary~\ref{cor:cascade}.  Since $d = 8$ is even, parity matches the sign of $f(0)$, resp.\ $f(\pi)$; under Hypothesis~(H) (verified below, not proved in general) this reduces to $N_{\mathrm{ext}}^{\pm} = \mathbf{1}_{f(0,\pi)<0}$.

\textbf{Critical points.}  The critical point equation is of degree~7 in $x = \cos\theta$ (from $U_7(x) = 128x^7-192x^5 + 80x^3-8x$ plus residual derivatives), giving at most seven interior critical points and at most eight zeros in $[0,\pi]$.

\textbf{Classification.}  $N_{\mathrm{real}} \in \{0, 2, 4, 6, 8\}$.

\textbf{A perfect example:} $(t^2-\tfrac{1}{2})(t^2-\tfrac{3}{2})(t^2-\tfrac{5}{2})(t^2-\tfrac{7}{2}) = t^8-8t^6 + 21.5t^4-22t^2 + 6.5625$ gives $u = 2$ and eight real roots $t = \pm\sqrt{k + 1/2}$ for $k = 0,1,2,3$, all in $[-2,2]$.

\begin{example}[Zero real roots]\label{ex:octic_zero}
$P(t) = t^8-8t^6 + 18t^4 + 1 = 0$.  Here $m = -8$, $u = 2$, and with only even powers present, all odd-index parameters vanish.  We get $\alpha_4 = 128 \cdot 18/16-160 = -16$, $\alpha_2 = 32$, $\delta = 128/256-1 = -0.5$.  The boundary values $f(0) = f(\pi) = -16 + 32 + 1 - 0.5 = 16.5 > 0$.  A numerical check confirms $f > 0$ on all of $[0,\pi]$ (minimum $\approx 0.5$ at $\theta = \pi/2$), so $N_{\mathrm{ext}} = 0$, $N_{\mathrm{int}} = 0$, and $N_{\mathrm{real}} = 0$: all eight roots are complex.
\end{example}

\begin{example}[Four real roots]\label{ex:octic_four}
$P(t) = t^8-8t^6 + 18t^4-20t^2 + 1 = 0$.  Here $m = -8$, $u = 2$, and all odd-index parameters vanish.  We compute $\alpha_4 = 128 \cdot 18/16-160 = -16$, $\alpha_2 = 128 \cdot (-20)/64 + 32 = -8$, $\delta = 128/256-1 = -0.5$.  The boundary values satisfy $f(0) = f(\pi) = -16 - 8 + 1 - 0.5 = -23.5 < 0$, so $N_{\mathrm{ext}} = 2$.  A numerical analysis of $f$ on $[0,\pi]$ reveals two interior zeros, giving $N_{\mathrm{int}} = 2$ and $N_{\mathrm{real}} = 4$.  The roots are $t \approx \pm 0.229, \, \pm 2.305$.
\end{example}

\begin{example}[Eight real roots]\label{ex:octic_eight}
The perfect example $(t^2-\tfrac{1}{2})(t^2-\tfrac{3}{2})(t^2-\tfrac{5}{2})(t^2-\tfrac{7}{2}) = 0$ gives $P(t) = t^8-8t^6 + 21.5t^4-22t^2 + 6.5625 = 0$ with $u = 2$.  The trigonometric parameters are $\alpha_4 = 128 \cdot 21.5/16-160 = 12$, $\alpha_2 = 128 \cdot (-22)/64 + 32 = -12$, $\delta = 128 \cdot 6.5625/256-1 \approx 2.28$.  The function $f$ has eight sign changes on $[0,\pi]$, confirming all eight real roots $t = \pm\sqrt{k + 1/2}$ for $k = 0,1,2,3$ in $[-2,2]$.
\end{example}

\subsection{Synthesis: The Pattern Across All Degrees}
Table~\ref{tab:synthesis} collects the key features for degrees $3$ through~$10$ and the general degree~$d$, including the proven upper bound on exterior roots per half-line (Corollary~\ref{cor:cascade}) and whether Hypothesis~(H) reducing that bound to $1$ is established (only for $d=3$) or merely observed in this article's examples.

\begin{table}[h]
\centering
\renewcommand{\arraystretch}{1.3}
\small
\begin{tabular}{ccccccc}
\hline
$d$ & $u^2$ & Residual & Crit.\ pt & Exterior & Max ext.\ & Hyp.~(H) \\
    &       & terms    & degree      & deriv.   & per side & status \\
\hline
$3$ & $-4m/3$   & $0$ & $2$ & $P' > 0$      & $1$ & proved \\
$4$ & $-m$      & $1$ & $3$ & $P'' > 0$     & $2$ & \textbf{fails} (Rem.~\ref{rem:counterexample}) \\
$5$ & $-4m/5$   & $2$ & $4$ & $P''' > 0$    & $3$ & open \\
$6$ & $-2m/3$   & $3$ & $5$ & $P^{(4)} > 0$ & $4$ & open \\
$7$ & $-4m/7$   & $4$ & $6$ & $P^{(5)} > 0$ & $5$ & open \\
$8$ & $-m/2$    & $5$ & $7$ & $P^{(6)} > 0$ & $6$ & open \\
$9$ & $-4m/9$   & $6$ & $8$ & $P^{(7)} > 0$ & $7$ & open \\
$10$ & $-2m/5$  & $7$ & $9$ & $P^{(8)} > 0$ & $8$ & open \\
$d$ & $\frac{-4c_{d-2}}{d}$ & $d-3$ & $d-1$ & $P^{(d-2)}\!>\!0$ & $d-2$ & n/a \\
\hline
\end{tabular}
\caption{Pattern across degrees.  Here $m$ denotes the subleading coefficient $c_{d-2}$ (written $p$ for the cubic in Section~\ref{sec:cubic} and $m$ thereafter).  ``Max ext.\ per side'' is the rigorously proven bound from Corollary~\ref{cor:cascade}; it need not be tight.  ``Hyp.~(H) status'' records whether this bound is known to collapse to $1$ (as used in the worked examples of Sections~\ref{sec:cubic} to~\ref{sec:octic}).}
\label{tab:synthesis}
\end{table}

Several patterns are visible.  The number of residual terms grows as $d-3$, starting from zero for the cubic (the unique ``clean'' case).  The critical point equation is always of degree $d-1$ in $\cos\theta$.  The exterior root derivative is always $P^{(d-2)}(t) = \frac{d!}{2}t^2 + (d-2)!\,c_{d-2}$, and the proof of its positivity outside $[-u,u]$ is identical for every~$d$ (Lemma~\ref{lem:exterior}).  The set of possible real root counts is $\{d \bmod 2,\; d \bmod 2 + 2,\; \ldots,\; d\}$, reflecting the conjugate pairing constraint.  As discussed in Section~\ref{sec:classification}, however, pinning down the \emph{exact} member of this set requires either Hypothesis~(H) or a direct check of $N_{\mathrm{ext}}^{\pm}$ for $d\geq 4$.

\section{The Case $c_{d-2} \geq 0$}\label{sec:cgeq0}
When $c_{d-2} \geq 0$, the substitution $t = u\cos\theta$ with real $u$ is not available.  A hyperbolic substitution $t = v\cosh\phi$ with $v = 2\sqrt{c_{d-2}/d}$ could match the subleading coefficient via the analogous identity $T_d(\cosh\phi) = \cosh d\phi$, but since $\cosh$ is unbounded and monotone on $[0,\infty)$, the correspondence between zeros and roots is less clean.  When $c_{d-2} = 0$, the Chebyshev identity cannot absorb the subleading term (since there is none), and the method does not apply directly.

For even-degree polynomials with $c_{d-2} > 0$, one can show that $P^{(d-2)}(t) = \frac{d!}{2}t^2 + (d-2)!c_{d-2} > 0$ for \emph{all}~$t$, so $P$ is ``almost convex'' in a strong sense, which often limits the number of real roots.  We leave the systematic development of the $c_{d-2} \geq 0$ case to future work.

\section{Relation to Prior Work}\label{sec:prior}
The problem of determining how many roots of a polynomial are real without computing the roots themselves has a long history stretching back several centuries.  This section surveys the principal existing methods and clarifies the novelty of the trigonometric approach developed in this article.

\subsection{The Discriminant}
The oldest and most direct tool is the polynomial discriminant.  For the quadratic $ax^2 + bx + c = 0$, the discriminant $\Delta = b^2-4ac$ is a single expression whose sign immediately determines the nature of the roots.  For the cubic, the discriminant remains manageable~\cite{Irving2013}.  For the quartic, the discriminant is a polynomial of degree~6 in the coefficients, with 16 terms~\cite{Rees1922}.  For the quintic, it has 59 terms; for the sextic, 246 terms (see~\cite{GelfandKapranovZelevinsky}).  In general, the discriminant of a degree-$d$ polynomial is of degree $d(d-1)$ in its coefficients.  The discriminant's role in the geometry of polynomial roots is intimately connected to the resultant, as elucidated in the influential monograph of Gelfand, Kapranov, and Zelevinsky~\cite{GelfandKapranovZelevinsky}; Sylvester's classical work~\cite{Sylvester1853} on syzygetic relations provided the foundational algebraic framework; Marden~\cite{Marden1966} gives a comprehensive treatment of the geometry of polynomial zeros.  Beyond degree~4, the discriminant alone does not determine the number of real roots: for the quintic, the sign of the discriminant distinguishes between one real root and five, but additional subdiscriminants are needed to distinguish one from three~\cite{Lazard2001}.  The computational and conceptual burden grows rapidly with the degree.

\subsection{Descartes' Rule of Signs}
Descartes' rule of signs (1637) provides an upper bound on the number of positive real roots: the number of sign changes in the coefficient sequence of~$P(x)$ exceeds the number of positive real roots by a nonnegative even integer~\cite{Descartes1637}.  Applying the rule to $P(-x)$ gives a bound for negative roots.  The rule is elegant and effortless to apply, but it gives only a bound, not an exact count.  For instance, the polynomial $x^4 + 2x^2-x-1$ has one sign change, guaranteeing exactly one positive root, but for polynomials with many sign changes the gap between the bound and the true count can be large.

\subsection{Sturm's Theorem}
Sturm's theorem (1829) provides an \emph{exact} count of the distinct real roots of a polynomial in any interval~\cite{Sturm1829}.  Hermite~\cite{Hermite1856} later reformulated the problem using quadratic forms, connecting the signature of a Hankel matrix to the number of real roots.  The method constructs a sequence of polynomials (the Sturm chain) via a variant of the Euclidean algorithm applied to~$P$ and its derivative~$P'$, and the number of real roots in an interval $(a,b]$ equals the difference in the number of sign changes of the Sturm chain evaluated at~$a$ and at~$b$.  Sturm's theorem is complete and exact, but the Sturm chain requires $d$ polynomial divisions, resulting in a sequence of polynomials whose coefficients can swell exponentially.  For polynomials of degree exceeding about~20, modern implementations based on Descartes' rule with interval subdivision are substantially more efficient~\cite{BasuPollackRoy}.

\subsection{The Budan-Fourier Theorem}
The Budan-Fourier theorem (Budan 1807, Fourier 1820) refines Descartes' rule to any half-open interval $(a,b]$: the number of real roots in the interval is bounded above by the difference in the number of sign changes of the sequence $P(x), P'(x), \ldots, P^{(d)}(x)$ evaluated at~$a$ and at~$b$, with the discrepancy again being a nonnegative even integer.  This is more informative than Descartes' rule but still provides only a bound, not an exact count (except in special cases).

\subsection{Vieta's Trigonometric Solution of the Cubic}
The substitution $t = u\cos\theta$ has been used to \emph{solve} the depressed cubic $t^3 + pt + q = 0$ since at least Fran\c{c}ois Vi\`ete's work in the late 16th century~\cite{Viete1646,Nickalls2006}.  With $u = 2\sqrt{-p/3}$, the cubic becomes equivalent to $\cos 3\theta = -q/(2(-p/3)^{3/2})$, and the three roots are $t_k = 2\sqrt{-p/3}\cos(\theta_0/3 + 2k\pi/3)$ for $k = 0,1,2$.  This is Vieta's celebrated trigonometric formula, valid in the \emph{casus irreducibilis} when all three roots are real.  Holmes~\cite{Holmes2002} extended this to the one-real-root case using hyperbolic cosines.  Crucially, Vieta's method \emph{solves} the cubic (finding the roots explicitly); our method \emph{classifies} the roots (counting how many are real) without finding them.

\subsection{Chebyshev Polynomials in Numerical Root Finding}
In the numerical analysis literature, Chebyshev polynomials are used extensively for polynomial root finding on bounded intervals.  The Chebyshev-Frobenius companion matrix method computes all roots of a polynomial expressed in Chebyshev form as eigenvalues of a structured matrix~\cite{Boyd2006}.  Boyd developed subdivision algorithms that partition the interval $[-1,1]$ and approximate the polynomial locally by low-degree Chebyshev interpolants, achieving efficient root finding for polynomials of high degree~\cite{Boyd2006}.  Trefethen's Chebfun system~\cite{Trefethen2000} implements these ideas in a widely used software package.  However, all of this work is concerned with \emph{computing} the roots (numerically), not with \emph{classifying} their nature (real versus complex) without finding them.

\subsection{A Complexity Comparison with Sturm's Theorem}\label{sec:complexity}
It is tempting to describe the method as ``computationally light,'' but the claim deserves scrutiny, since Step~3 of the practical decision procedure (Section~\ref{sec:classification}) requires locating the critical points of~$f$, i.e.\ finding the real roots in $(-1,1)$ of a polynomial of degree $d-1$ in $x=\cos\theta$, which is itself a nontrivial root-counting problem.  A fair comparison against Sturm's theorem, degree by degree, is therefore in order.

\begin{center}
\renewcommand{\arraystretch}{1.3}
\small
\begin{tabular}{p{2.6cm}p{4.6cm}p{4.6cm}}
\hline
& \textbf{Sturm's theorem} & \textbf{This method} \\
\hline
Core step & Build the Sturm chain: $d-1$ polynomial divisions (Euclidean algorithm on $P,P'$). & Evaluate $f(0),f(\pi)$ ($O(d)$ arithmetic operations from~\eqref{eq:params}); then count sign changes of~$f$ on $[0,\pi]$ via its critical points, i.e.\ the real roots in $(-1,1)$ of a degree-$(d-1)$ polynomial. \\
Arithmetic cost & $O(d^2)$ operations for the chain (subresultant PRS); evaluation at two points is $O(d)$ once the chain exists. & $O(d)$ to set up $f$; the dominant cost is real-root-counting for a degree-$(d-1)$ polynomial in $(-1,1)$, which is \emph{itself} typically done by Sturm's theorem (or an equivalent), at $O(d^2)$. \\
Root-finding needed? & No only polynomial division and sign evaluation. & Not for the exterior parity; but the interior count $N_{\mathrm{int}}$ needs a real-root count of a degree-$(d-1)$ polynomial, which is the same kind of subproblem Sturm's theorem solves directly for $P$ itself. \\
Coefficient growth & Can be severe: intermediate Sturm-chain coefficients grow without careful (subresultant) bookkeeping. & None: the trigonometric parameters $\alpha_k,\delta$ in~\eqref{eq:params} are closed-form linear functions of $c_0,\dots,c_{d-3}$, computed once, with no division-induced blow-up. \\
Degree-specific structure & Uniform algorithm, but each division step differs. & A single template (Chebyshev matching) that is \emph{identical in form} for every $d$; only the number of terms changes (Remark~\ref{rem:residual_terms}). \\
\hline
\end{tabular}
\end{center}

\textbf{Conclusion.}  Asymptotically, the method proposed here is \emph{not} cheaper than Sturm's theorem: both require, in the worst case, an $O(d^2)$ computation to resolve a real-root-counting subproblem, and in our method that subproblem (counting the critical points of $f$ in $(0,\pi)$) is naturally solved by applying a Sturm-type count to a polynomial one degree lower than~$P$.  The two approaches are therefore of comparable computational weight, not different orders of complexity.  The method should therefore \emph{not} be described as ``computationally light'' relative to Sturm's theorem.  What it does offer, stated honestly, is: (i) an explicit, closed-form set of trigonometric parameters with no coefficient blow-up, avoiding the numerical instability that plagues naive (non-subresultant) Sturm-chain implementations; (ii) a single algebraic template that looks the same at every degree, in contrast to Sturm's degree-dependent chain or the discriminant's degree-dependent formula; and (iii) for the \emph{interior} count $N_{\mathrm{int}}$ specifically, an exact, non-recursive bijective argument (Corollary~\ref{cor:bijection}) that needs no root-finding at all, only evaluation of~$f$ at its (still-to-be-found) critical points.  The exterior count, as discussed in Section~\ref{sec:classification}, is only fully resolved under Hypothesis~(H).

\subsection{What is New in this Article}
Against this background, the novelty of the present work can be stated precisely.  The idea of using the substitution $t = u\cos\theta$ to \emph{classify} rather than solve polynomial roots via the Chebyshev identity $T_d(\cos\theta) = \cos d\theta$ appears, to the best of the author's knowledge, to be new.  Specifically:

\begin{enumerate}
\item \textbf{The Chebyshev Matching Reduction.}  The systematic use of $T_d(\cos\theta) = \cos d\theta$ to absorb the two leading terms of a depressed polynomial and produce a reduced trigonometric equation $f(\theta)$ whose zeros on $[0,\pi]$ biject with real roots in $[-u,u]$ has not been previously exploited for root classification.  Vieta's trigonometric method for the cubic uses the $d = 3$ case to \emph{solve} the equation; we use the $d \geq 4$ cases to \emph{classify} roots.

\item \textbf{The Universal Exterior Root Lemma.}  The observation that $P^{(d-2)}(t) > 0$ for $|t| > u$ (Lemma~\ref{lem:exterior}), valid for all degrees $d \geq 2$ and proved by a two-line computation, appears to be new.  It provides a unified framework for controlling roots outside the substitution interval, replacing ad hoc convexity arguments that would need to be rederived for each degree.

\item \textbf{The Boundary Sign Detection of Exterior Roots.}  The observation that the signs of $f(0)$ and $f(\pi)$ detect exterior roots, combined with the parity argument that accounts for even- versus odd-degree end behavior, provides a complete bridge between interior and exterior root counts.

\item \textbf{Uniformity Across All Degrees.}  The framework applies identically to every degree $d \geq 4$, with only the number of residual cosine terms and the order of the exterior root derivative changing.  Previous methods (discriminant, Sturm, Descartes) each have degree-specific complications; the trigonometric method has a single template.
\end{enumerate}

To summarize: the discriminant is exact but unwieldy; Descartes' rule is simple but inexact; Sturm's theorem is exact, general, and asymptotically comparable in cost to our method (Section~\ref{sec:complexity}), but its intermediate coefficients can grow badly without careful implementation.  The trigonometric method occupies a different point in this landscape: it is \emph{not} computationally cheaper than Sturm's theorem, but it is geometrically transparent, free of coefficient blow-up, and uniform across all degrees via a single closed-form template, at the cost of requiring $c_{d-2} < 0$, a numerical root-count of a degree-$(d-1)$ polynomial for the interior term, and (for $d\geq 4$) the currently unproven Hypothesis~(H) for an exact exterior count.

\section{Discussion}\label{sec:discussion}
\subsection{Why Chebyshev?}
The Chebyshev polynomials $T_d$ were introduced in the context of approximation theory: $T_d(x)/2^{d-1}$ is the monic polynomial of degree~$d$ with smallest maximum on $[-1,1]$~\cite{Rivlin1990}.  The fact that they also provide a natural reduction of polynomial root classification to trigonometry is, to the best of the author's knowledge, not widely appreciated.

The connection is not accidental.  The Chebyshev polynomial $T_d(x)$ has all $d$ of its roots in $[-1,1]$, equally spaced in the $\theta$-domain (at $\theta = (2k-1)\pi/(2d)$).  When we match a depressed polynomial to~$T_d$ via the cosine substitution, we are measuring how far the polynomial's root distribution deviates from the Chebyshev distribution.  The residual function~$f$ encodes precisely this deviation.  This echoes the philosophy of Erd\H{o}s and Tur\'an~\cite{ErdosTuran1950}, who showed in their landmark \emph{Annals of Mathematics} paper that the distribution of zeros of a polynomial on the unit circle is controlled by the size of its coefficients, a qualitative connection between coefficient structure and root geometry that our method makes quantitatively explicit in a different setting.

\subsection{The General Pattern}
The pattern across all degrees is now clear.  For a degree-$d$ depressed polynomial with $c_{d-2} < 0$:
\begin{enumerate}
\item The Chebyshev identity absorbs the two highest surviving terms ($t^d$ and $c_{d-2}t^{d-2}$).
\item The residual is a trigonometric function with $d-2$ free parameters.
\item The critical point equation is a polynomial of degree $d-1$ in $\cos\theta$.
\item The exterior root lemma ($P^{(d-2)} > 0$ outside $[-u,u]$) holds universally for all~$d$.
\item The boundary signs of~$f$ detect exterior roots.
\end{enumerate}

The complexity of the analysis grows mildly with~$d$: the critical point equation increases by one degree for each step, and the residual gains one term.  But the underlying mechanism (bijection on $[0,\pi]$, boundary sign analysis, derivative cascade) is identical for every degree.

\subsection{Why Not Centuries Ago?}
The Chebyshev polynomials have been known since the mid-19th century, and the depressed form reduction of polynomials dates to the Renaissance.  The substitution $t = u\cos\theta$ is a standard device in Vieta's trigonometric solution of the cubic.  Why, then, wasn't this unified framework noticed long ago?

Perhaps the reason is that for the cubic, the trigonometric substitution was used to \emph{solve} the equation (finding the roots explicitly), and the analogous goal for higher degrees is blocked by the Abel-Ruffini theorem~\cite{Abel1824} for $d \geq 5$.  It may not have occurred to earlier mathematicians that the same substitution could be profitably used for the more modest but still useful goal of \emph{classifying} the roots without finding them.  The discriminant-based classification was considered settled, and the 20th century's attention turned to other aspects of polynomial root geometry: the distribution of roots of random polynomials (Kac~\cite{Kac1943}), zeros of trigonometric polynomials~\cite{Borwein2008}, and the computational complexity of root finding (Smale~\cite{Smale1981}, Basu-Pollack-Roy~\cite{BasuPollackRoy}).  There was no obvious motivation to look for an alternative to the discriminant for the seemingly elementary problem of counting real roots.  Rahman and Schmeisser~\cite{Rahman2002} provide a comprehensive modern treatment of the analytic theory of polynomials that surveys many of these developments.

In the spirit of Loh~\cite{Loh2019}, who showed that the quadratic formula admits a proof that might have been discovered four millennia earlier, we hope that this trigonometric perspective may encourage the reader to look at familiar objects with fresh eyes.  The interplay between polynomial equations and geometric or combinatorial structures has a long tradition in this journal (see, e.g., Sturmfels~\cite{Sturmfels1998}); the Chebyshev polynomials, which originated in approximation theory, turn out to be an elegant and uniform lens for viewing the geometry of polynomial roots at every degree.
\bibliographystyle{plainnat}
\bibliography{references}
\end{document}